\documentclass[11pt]{amsart}
\usepackage{amssymb}
\usepackage{mathrsfs}
\usepackage{enumerate}
\makeatletter
\@namedef{subjclassname@2010}{%
  \textup{2010} Mathematics Subject Classification}
\makeatother
\newtheorem{lemma}{Lemma}
\newtheorem{theorem}{Theorem}[section]
\newtheorem{pro}{Proposition}
\newtheorem{corollary}{Corollary}
\newtheorem{defi}{Definition}
\newtheorem{remark}{Remark}

\begin{document}

\title{WEIGHTED RIESZ BASES IN G-FUSION FRAMES AND THEIR PERTURBATION}

\author[Gh. Rahimlou]{GHOLAMREZA RAHIMLOU}
\address{Department of Mathematics, Faculty of Tabriz  Branch\\ Technical and Vocational University (TUV), East Azarbaijan
, Iran\\}
\email{grahimlou@gmail.com}

\author[V. Sadri]{Vahid Sadri}
\address{Department of Mathematics, Faculty of Tabriz  Branch\\ Technical and Vocational University (TUV), East Azarbaijan
, Iran}
\email{vahidsadri57@gmail.com}

\author[R. Ahmadi]{Reza Ahmadi}
\address{Institute of Fundamental Sciences\\University of Tabriz\\, Iran\\}
\email{rahmadi@tabrizu.ac.ir}


\begin{abstract}
In this paper, we introduce orthonoramal and Riesz bases for g-fusion frames and will show that the weights have basic roles. Next, we prove an effective theorem between frames and g-fusion frames by using an operator. Finally, perturbations of g-fusion frames will be presented.
\end{abstract}



\maketitle

\section{Introduction}
Bases play a prominent role in discrete frames and their studying can extract interesting properties from the frames. One of the most important types of bases are  orthonormal  bases and also, as their special case, the Riesz basis. The Riesz basis has been defined in \cite{ch} by the image of  orthonormal bases with a bounded bijective operator. Afterwards, this basis has been equalized in a property with complete sequences and  inequalities for the synthesis operator. Sun in \cite{sun} could introduce a Riesz basis for g-frames by using that property and we will continue his method in Section 3 for g-fusion frames. In Section 4, we present a useful operator for characterizations of these frames and finally in Section 5, a perturbation of these frames will be studied.

\section{Preliminaries}
Throughout this paper, $H$ and $K$ are separable Hilbert spaces and $\mathcal{B}(H,K)$ is the collection of all bounded linear operators of $H$ into $K$. If $K=H$, then $\mathcal{B}(H,H)$ will be denoted by $\mathcal{B}(H)$. Also, $\pi_{V}$ is the orthogonal projection from $H$ onto a closed subspace $V\subset H$ and  $\lbrace H_j\rbrace_{j\in\Bbb J}$ is a sequence of Hilbert spaces, where $\Bbb J$ is a subset of $\Bbb Z$. It is easy to check that if $u\in\mathcal{B}(H)$ is an invertible operator, then 
$$\pi_{uV}u\pi_{V}=u\pi_{V}.$$
 We define the space $\mathscr{H}_2:=(\sum_{j\in\Bbb J}\oplus H_j)_{\ell_2}$ by
\begin{eqnarray}
\mathscr{H}_2=\big\lbrace \lbrace f_j\rbrace_{j\in\Bbb J} \ : \ f_j\in H_j , \ \sum_{j\in\Bbb J}\Vert f_j\Vert^2<\infty\big\rbrace,
\end{eqnarray}
with the inner product defined by
$$\langle \lbrace f_j\rbrace, \lbrace g_j\rbrace\rangle=\sum_{j\in\Bbb J}\langle f_j, g_j\rangle.$$
It is clear that $\mathscr{H}_2$ is a Hilbert space with pointwise operations.
\begin{defi}
Let $W=\lbrace W_j\rbrace_{j\in\Bbb J}$ be a collection of closed subspaces of $H$, $\lbrace v_j\rbrace_{j\in\Bbb J}$ be a family of weights, i.e. $v_j>0$  and $\Lambda_j\in\mathcal{B}(H,H_j)$ for each $j\in\Bbb J$. We say $\Lambda:=(W_j, \Lambda_j, v_j)$ is a generalized fusion frame (or g-fusion frame) for $H$ if there exists $0<A\leq B<\infty$ such that for each $f\in H$
\begin{eqnarray}\label{g}
A\Vert f\Vert^2\leq\sum_{j\in\Bbb J}v_j^2\Vert \Lambda_j \pi_{W_j}f\Vert^2\leq B\Vert f\Vert^2.
\end{eqnarray} 
\end{defi}
We call $\Lambda$ a Parseval g-fusion frame if $A=B=1$. When the right hand side of (\ref{g}) holds, $\Lambda$ is called a g-fusion Bessel sequence for $H$ with bound $B$. Throughout this paper, $\Lambda$ will be a triple $(W_j, \Lambda_j, v_j)$ with $j\in\Bbb J$ unless otherwise noted.

The synthesis and analysis operators in the g-fusion frames are defined by
\begin{align*}
T_{\Lambda}&:\mathscr{H}_2\longrightarrow H \ \ \ \ \ \ \ \ \ , \ \ \ \ \ \ T_{\Lambda}^*:H\longrightarrow\mathscr{H}_2\\
T_{\Lambda}(\lbrace f_j\rbrace_{j\in\Bbb J})&=\sum_{j\in\Bbb J}v_j \pi_{W_j}\Lambda_{j}^{*}f_j \ \ \ , \ \ \ \ T_{\Lambda}^*(f)=\lbrace v_j \Lambda_j \pi_{W_j}f\rbrace_{j\in\Bbb J}.
\end{align*}
Thus, the g-fusion frame operator is given by
$$S_{\Lambda}f=T_{\Lambda}T^*_{\Lambda}f=\sum_{j\in\Bbb J}v_j^2 \pi_{W_j}\Lambda^*_j \Lambda_j \pi_{W_j}f$$
and
\begin{equation}\label{sf1}
\langle S_{\Lambda}f, f\rangle=\sum_{j\in\Bbb J}v_j^2\Vert \Lambda_j \pi_{W_j}f\Vert^2.
\end{equation}
Therefore
$$A I\leq S_{\Lambda}\leq B I.$$
This means that $S_{\Lambda}$ is a bounded, positive and invertible operator (with adjoint inverse). So, we have the reconstruction formula for any $f\in H$:
\begin{equation}\label{3}
f=\sum_{j\in\Bbb J}v_j^2 \pi_{W_j}\Lambda^*_j \Lambda_j \pi_{W_j}S^{-1}_{\Lambda}f
=\sum_{j\in\Bbb J}v_j^2 S^{-1}_{\Lambda}\pi_{W_j}\Lambda^*_j \Lambda_j \pi_{W_j}f.
\end{equation}
For the proof of followings, see \cite{sad}.
\begin{theorem}\label{2.2}
$\Lambda$ is a g-fusion Bessel sequence for $H$ with bound $B$ if and only if the operator $T_{\Lambda}$ is  well-defined and bounded operator with $\Vert T_{\lambda}\Vert\leq \sqrt{B}$.
\end{theorem}
\begin{theorem}\label{2.3}
$\Lambda$ is a g-fusion frame for $H$ if and only if 
 \begin{align*}
T_{\Lambda}&:\mathscr{H}_2\longrightarrow H,\\
T_{\Lambda}(\lbrace f_j\rbrace_{j\in\Bbb J})&=\sum_{j\in\Bbb J}v_j \pi_{W_j}\Lambda_{j}^{*}f_j
\end{align*}
is a well-defined, bounded and surjective.
\end{theorem}
\begin{defi}
A g-fusion frame $\tilde{\Lambda}:=(S^{-1}_{\Lambda}W_j, \Lambda_j\pi_{W_j} S_{\Lambda}^{-1}, v_j)$ with g-fusion frame operator $S_{\tilde{\Lambda}}=T_{\tilde{\Lambda}}T^*_{\tilde{\Lambda}}$ is called the \textit{(canonical) dual g-fusion frame}  of $\Lambda$.
\end{defi}
Now, we can obtain
\begin{align}\label{frame}
f=\sum_{j\in\Bbb J}v_j^2\pi_{W_j}\Lambda^*_j\tilde{\Lambda_j}\pi_{\tilde{W_j}}f=
\sum_{j\in\Bbb J}v_j^2\pi_{\tilde{W_j}}\tilde{\Lambda_j}^*\Lambda_j\pi_{W_j}f,
\end{align}
where $\tilde{W_j}:=S^{-1}_{\Lambda}W_j  \ , \ \tilde{\Lambda_j}:=\Lambda_j \pi_{W_j}S_{\Lambda}^{-1}.$
\begin{defi}
Let $\Lambda_j\in\mathcal{B}(H,H_j)$ for each $j\in\Bbb J$.  $\Lambda=(W_j, \Lambda_j, v_j)$ is called gf-complete, if
$$\overline{\mbox{span}}\lbrace \pi_{W_j}\Lambda^*_j H_j\rbrace=H.$$
\end{defi}
It is easy to check that $\Lambda$ is gf-complete if and only if 
$$\lbrace f: \ \Lambda_j \pi_{W_j}f=0 , \ j\in\Bbb J\rbrace=\lbrace 0\rbrace.$$
\begin{pro}\label{p3}
If $\Lambda=(W_j, \Lambda_j, v_j)$ is a g-fusion frame for $H$, then $\Lambda$ is a complete.
\end{pro}
\begin{proof}
See \cite{sad}.
\end{proof}
\section{gf-Riesz and Orthonormal Bases}
\begin{defi}
 Let $W=\lbrace W_j\rbrace_{j\in\Bbb J}$ be a collection of closed subspaces of $H$ and $j\in\Bbb J$.
We say that $(W_j, \Lambda_j)$ is a gf-orthonormal bases for $H$ with respect to $\{v_j\}_{j\in\Bbb J}$, if 
\begin{align}\label{o1}
\langle v_i\pi_{W_i}\Lambda^*_i g_i, v_j\pi_{W_j}\Lambda^*_{j}g_j\rangle=\delta_{i,j}\langle g_i, g_j\rangle \ , \ \ \ i,j\in\Bbb J\ , \ \ \ g_i\in H_i\ , \ \ \ g_j\in H_j
\end{align}
\begin{align}\label{o2}
\sum_{j\in\Bbb J}v_j^2\Vert \Lambda_{j}\pi_{W_j}f\Vert^2=\Vert f\Vert^2 \ , \ \ \ \ f\in H.
\end{align}
\end{defi}
\begin{defi}
$\Lambda=(W_j, \Lambda_j, v_j)$ is called a gf-Riesz basis for $H$ if 
\begin{enumerate}
\item $\Lambda$ is  gf-complete,
\item There exist $0<A\leq B<\infty$ such that for each finite subset $\Bbb I\subseteq\Bbb J$ and $g_j\in H_j$, $j\in\Bbb I$,
\begin{equation}
A\sum_{j\in\Bbb I}\Vert g_j\Vert^2\leq\big\Vert\sum_{j\in\Bbb I}v_j\pi_{W_j}\Lambda^*_j g_j \big\Vert^2\leq B\sum_{j\in\Bbb I}\Vert g_j\Vert^2.
\end{equation}
\end{enumerate}
\end{defi}
It is easy to check that if $\Lambda$ is a gf-Riesz bases for $H$, then the operator $T_{\Lambda}$ which is defined by
\begin{align*}
T_{\Lambda}&:\mathscr{H}_2\longrightarrow H\\
T_{\Lambda}(\lbrace g_j\rbrace_{j\in\Bbb J})&=\sum_{j\in\Bbb J}v_j \pi_{W_j}\Lambda_{j}^{*}g_j,
\end{align*}
is injetive.
\begin{pro}
Let $\Lambda=(W_j, \Lambda_j, v_j)$ be a g-fusion frame for $H$ and suppose that (\ref{o1}) holds. Then $\Lambda$ is a gf-orthonormal basis for $H$.
\end{pro}
\begin{proof}
Assume that $S_{\Lambda}$ is the g-fusion frame operator of $\Lambda$ and
$$M:=\lbrace f\in H\ : \ S_{\Lambda}f=f\rbrace.$$
It is clear that $M$ is a non-empty closed subspace of $H$. Let $f\in H$ and $k\in\Bbb J$. Since
$$v_j v_k\pi_{W_j}\Lambda^*_j\Lambda_j\pi_{W_j}\pi_{W_k}\Lambda^*_{k}\Lambda_k\pi_{W_k}f=\delta_{j,k}\pi_{W_k}\Lambda^*_{k}\Lambda_k\pi_{W_k}f,$$
then $\pi_{W_k}\Lambda^*_{k}\Lambda_k\pi_{W_k}f\in M$. So, for any $h\in M^{\perp}$ and $g\in M$ we have, for all $j\in\Bbb J$
$$\langle\pi_{W_j}\Lambda^*_j\Lambda_j\pi_{W_j}h, g\rangle=\langle h, \pi_{W_j}\Lambda^*_j\Lambda_j\pi_{W_j}g\rangle=0.$$
Thus $\pi_{W_j}\Lambda^*_j\Lambda_j\pi_{W_j}h=0$ for each $j\in\Bbb J$, and so $\Vert\Lambda_j\pi_{W_j}h\Vert=0$. By definition of g-fusion frame, we obtain $h=0$. Therefore, $M^{\perp}=\lbrace0\rbrace$ and we conclude $H=M$. So, $S_{\Lambda}=id_{H}$ and the proof is completed.
\end{proof}
\begin{theorem}
$\Lambda$ is a gf-orthonormal bases for $H$ if and only if
\begin{enumerate}
\item[(I)] $v_j\pi_{W_j}\Lambda^*_j$ is isometric for any $j\in\Bbb J$;
\item[(II)] $\bigoplus_{j\in\Bbb J}v_j\pi_{W_j}\Lambda^*_j(H_j)=H$.
\end{enumerate}
\end{theorem}
\begin{proof}
Suppose that $\Lambda$ is a gf-orthonormal basis for $H$. If $j\in\Bbb J$ and $g\in H_j$, we have
$$\langle v_j\pi_{W_j}\Lambda^*_jg, v_j\pi_{W_j}\Lambda^*_jg\rangle=\delta_{j,j}\langle g, g\rangle=\langle g, g\rangle.$$
Then, $v_j\pi_{W_j}\Lambda^*_j$ is isometric for any $j\in\Bbb J$. So, $v_j\pi_{W_j}\Lambda^*_j(H_j)$ is a closed subspace of $H$. Therefore, for each $j\in\Bbb J$ and $f\in H$, we have from (\ref{o2})
\begin{align*}
\langle f, f\rangle&=\sum_{j\in\Bbb J}v_j^2\langle \Lambda_j\pi_{W_j}f, \Lambda_j\pi_{W_j}f\rangle\\
&=\langle\sum_{j\in\Bbb J}v_j^2 \pi_{W_j}\Lambda^*_j\Lambda_j\pi_{W_j}f, f\rangle.
\end{align*}
Thus, for each $f\in H$,
$$f=\sum_{j\in\Bbb J} v_j^2\pi_{W_j}\Lambda^*_j\Lambda_j\pi_{W_j}f.$$
By letting $g_j:=v_j\Lambda_j\pi_{W_j}f$, we obtain 
$f=\sum_{j\in\Bbb J}v_j \pi_{W_j}\Lambda^*_jg_j$ and
$$\sum_{j\in\Bbb J}v_j^2\Vert\pi_{W_j}\Lambda^*_jg_j\Vert^2=\sum_{j\in\Bbb J}\Vert g_j\Vert^2=\sum_{j\in\Bbb J}v_j^2\Vert\Lambda_j\pi_{W_j}f\Vert^2=\Vert f\Vert^2.$$
So, $\bigoplus_{j\in\Bbb J}v_j\pi_{W_j}\Lambda^*_j(H_j)=H$. Conversely, if (I) , (II) are satisfied, then (\ref{o1}) is clear. Indeed, for each $i\neq j$, $v_i\pi_{W_i}\Lambda^*_i\perp v_j\pi_{W_j}\Lambda_j$ and $v_j\pi_{W_j}\Lambda^*_j$ is isometric. Let $f\in H$, we get from (II) for any $j\in\Bbb J$ and some $g_j\in H_j$
$$f=\sum_{j\in\Bbb J}v_j\pi_{W_j}\Lambda^*_jg_j$$  
and
$$\Vert f\Vert^2=\sum_{j\in\Bbb J}v_j^2\Vert\pi_{W_j}\Lambda^*_jg_j\Vert^2=\sum_{j\in\Bbb J}\Vert g_j\Vert^2.$$
Now, let $i\in\Bbb J$, then for each $f,h\in H_j$,
\begin{align*}
\langle v_j\Lambda_i\pi_{W_i}f, h\rangle&=\langle\sum_{j\in\Bbb J}v_j^2\Lambda_i\pi_{W_i}\pi_{W_j}\Lambda^*_jg_j, h\rangle\\
&=\sum_{j\in\Bbb J}v_j^2\langle\pi_{W_j}\Lambda^*_jg_j, \pi_{W_i}\Lambda^*_ih\rangle\\
&=\langle v_i\pi_{W_i}\Lambda^*_i g_i, v_i\pi_{W_i}\Lambda^*_ih\rangle\\
&=\langle g_i, h\rangle.
\end{align*}
Hence $g_i=v_i\Lambda_i\pi_{W_i}f$ for each $i\in\Bbb J$. So, $f=\sum_{j\in\Bbb J}v_j^2\pi_{W_j}\Lambda^*_j\Lambda_j\pi_{W_j}f$ for all $f\in H$ and  (\ref{o2}) is proved.
\end{proof}
\begin{corollary}\label{c1}
Every gf-orthonormal basss for $H$ is a gf-Riesz bases for $H$ with bounds $A=B=1$.
\end{corollary}
\begin{theorem}
Let $\Theta=(W_j, \Theta_j)$ be a gf-orthonormal bases with respect to $\{v_j\}_{j\in\Bbb J}$ and $\Lambda=(W_j, \Lambda_j, v_j)$ be a g-fusion frame for $H$ with same weights. Then, there exists a operator $V\in\mathcal{B}(H)$ and surjective such that $\Lambda_j\pi_{W_j}=\Theta_j\pi_{W_j}V^*$ for all $j\in\Bbb J$.
\end{theorem}
\begin{proof}
Let
\begin{align*}
V&:H\longrightarrow H,\\
Vf&=\sum_{j\in\Bbb J}v_j^2\pi_{W_j}\Lambda^*_j\Theta_j\pi_{W_j}f.
\end{align*}
Then, $V$ is well-defined and bounded. Indeed, for each finite subset $\Bbb I\subseteq\Bbb J$ and $f\in H$,
\begin{align*}
\Vert Vf\Vert&=\sup_{\Vert h\Vert=1}\big\vert\big\langle\sum_{j\in\Bbb J}v_j^2\pi_{W_j}\Lambda^*_j\Theta_j\pi_{W_j}f, h\big\rangle\big\vert\\
&\leq\sqrt{B}\big(\sum_{j\in\Bbb J}v_j^2\Vert\Theta_j\pi_{W_j}f\Vert^2\big)^{\frac{1}{2}}\\
&=\sqrt{B}\Vert f\Vert,
\end{align*}
where $B$ is an upper g-fusion frame bound for $\Lambda$. Therefore, the series  is weakly unconditionally Cauchy and so unconditionally convergent in $H$ (see \cite{diestel} page 58) and also $\Vert V\Vert\leq\sqrt{B}$. Since $\Theta$ is a gf-orthonormal bases, then 
$$\Theta_j\pi_{W_j}\pi_{W_i}\Theta^*_ig=(v_i v_j)^{-1}\delta_{i,j}g$$ and
\begin{align*}
V\pi_{W_i}\Theta^*_ig&=\sum_{j\in\Bbb J}v_j^2\pi_{W_j}\Lambda^*_j\Theta_j\pi_{W_j}\pi_{W_i}\Theta^*_ig\\
&=\pi_{W_i}\Lambda^*_i g,
\end{align*}
for all $g\in H_j$ and $i\in\Bbb J$. Thus $\Lambda_j\pi_{W_j}=\Theta_j\pi_{W_j}V^*$. Now, we show that $V$ is  surjective. Assume that $f\in H$. By Theorem \ref{2.3}, there is $\lbrace g_j\rbrace_{j\in\Bbb J}\in\mathscr{H}_2$ such that $\sum_{j\in\Bbb J}v_j\pi_{W_j}\Lambda^*_jg_j=f$. Let $g:=T_{\Theta}(\lbrace g_j\rbrace_{j\in\Bbb J})$, thus
$$Vg=\sum_{j\in\Bbb J}Vv_j\pi_{W_j}\Theta^*_jg_j=\sum_{j\in\Bbb J}v_j\pi_{W_j}\Lambda^*_jg_j=f$$
and $V$ is surjective.
\end{proof}
\begin{corollary}
If $\Lambda$ is a Parseval g-fusion frame for $H$, then $V^*$ is isometric.
\end{corollary}
\begin{corollary}\label{c3}
If $\Lambda$ is a gf-Riesz bases for $H$, then $V$ is invertible.
\end{corollary}
\begin{proof}
Let $Vf=0$ and $f\in H$. Since $T_{\Lambda}$ is injective and 
$$Vf=\sum_{j\in\Bbb J}v_j^2\pi_{W_j}\Lambda^*_j\Theta_j\pi_{W_j}f=T_{\Lambda}T^*_{\Theta}f,$$
therefore, $T^*_{\Theta}f=0$. So, $\Vert f\Vert^2=\Vert T_{\Theta}f\Vert^2=0$, hence, $f=0$.
\end{proof}
\begin{corollary}
If $(W_j, \Lambda_j)$ is a gf-orthonormal bases for $H$ with respect to $\{v_j\}_{j\in\Bbb J}$, then $V$ is unitary.
\end{corollary}
\begin{proof}
By Corollaries \ref{c1} and \ref{c3}, the operator $V$ is invertible. Let $f\in H$, we obtain
$$\Vert f\Vert^2=\sum_{j\in\Bbb J}v_j^2\Vert\Lambda_j\pi_{W_j}f\Vert^2=\sum_{j\in\Bbb J}\Vert\Theta_j\pi_{W_j}V^*f\Vert^2=\Vert V^*f\Vert^2.$$Thus, $VV^*=id_H$ and this means that $V$ is unitary.
\end{proof}
\section{Characterizations of g-fusion frames, gf-Riesz and gf-orthonormal bases}
Sun in \cite{sun} showed that each g-frame for $H$ induces a sequence in $H$ dependent on the g-frame and he proved a useful theorem about them. In this section, we are going to express Sun's method for g-fusion frames.

Let $W=\lbrace W_j\rbrace_{j\in\Bbb J}$ be a family of closed subspaces of $H$, $\lbrace v_j\rbrace_{j\in\Bbb J}$ be a family of weights, $\Lambda_j\in\mathcal{B}(H,H_j)$ for each $j\in\Bbb J$ and $\lbrace e_{j,k}\rbrace_{k\in\Bbb K_{j}}$ be an orthonormal basis for $H_j$, where  $\Bbb K_j\subseteq\Bbb Z$ and $j\in\Bbb J$. Suppose that
\begin{align*}
\varphi:& h\longrightarrow \Bbb C,\\
\varphi(f)&=\langle v_j\Lambda_j\pi_{W_j}f, e_{j,k}\rangle.
\end{align*}
We have
$$\Vert\varphi f\Vert\leq v_j\Vert \Lambda_j\Vert \Vert f\Vert,$$
therefore, $\varphi$ is a bounded linear functional on $H$. Now, we can write
$$\langle v_j\Lambda_j\pi_{W_j}f, e_{j,k}\rangle=\langle f, v_j\pi_{W_j}\Lambda^*_je_{j,k}\rangle.$$
So, if 
\begin{equation}\label{s1}
u_{j,k}:=v_j\pi_{W_j}\Lambda^*_j e_{j,k} , \ \ \ j\in\Bbb J , \ k\in\Bbb K_j
\end{equation}
then $\langle f, u_{j,k}\rangle=\langle v_j\Lambda_j\pi_{W_j}f, e_{j,k}\rangle$ for all $f\in H$.
\begin{remark}
Using (\ref{s1}), we get for each $f\in H$
\begin{equation}\label{s2}
v_j\Lambda_j\pi_{W_j}f=\sum_{k\in\Bbb K_{j}}\langle f, u_{j,k}\rangle e_{j,k}.
\end{equation}
But
$$\sum_{k\in\Bbb K_{j}}\vert\langle f, u_{j,k}\rangle\vert^2\leq v_j^2\Vert \Lambda_j\Vert^2\Vert f\Vert^2.$$
Thus, $\lbrace u_{j,k}\rbrace_{k\in\Bbb K_{j}}$ is a Bessel sequence for $H$.
It follows that for each $f\in H$ and $g\in H_j$
\begin{align*}
\langle f, v_j\pi_{W_j}\Lambda^*_jg\rangle&=\langle v_j\Lambda_j\pi_{W_j}f, g\rangle\\
&=\sum_{k\in\Bbb K_{j}}\langle v_j\pi_{W_j}\Lambda_jf, e_{j,k}\rangle\langle e_{j,k}, g\rangle\\
&=\sum_{k\in\Bbb K_{j}}\langle f, u_{j,k}\rangle\langle e_{j,k}, g\rangle\\
&=\Big\langle f, \sum_{k\in\Bbb K_{j}}\langle g, e_{j,k}\rangle u_{j,k}\Big\rangle.
\end{align*}
Therefore,
\begin{eqnarray}\label{s3}
 v_j\pi_{W_j}\Lambda^*_jg=\sum_{k\in\Bbb K_{j}}\langle g, e_{j,k}\rangle u_{j,k}
\end{eqnarray}
for all $g\in H_j$.
\end{remark}
We say $\lbrace u_{j,k} \ : \ j\in\Bbb J , \ k\in\Bbb K_j\rbrace$  the sequence induced by $\Lambda$.
\begin{theorem}
Let $\Lambda=(W_j, \Lambda_j, v_j)$ and $u_{j,k}$ be defined by (\ref{s1}). Then we get the followings.
\begin{enumerate}
\item[(I)] $\Lambda$ is a g-fusion frame (resp. g-fusion Bessel sequence, Parseval g-fusion frame, gf-Riesz basis, gf-orthonormal basis) for $H$ if and only if $\lbrace u_{j,k} \ : \ j\in\Bbb J , \ k\in\Bbb K_j\rbrace$ is a frame (resp. Bessel sequence, Parseval  frame, Riesz basis, orthonormal basis) for $H$.
\item[(II)] The g-fusion operator for $\Lambda$ coincides with the frame operator for $\lbrace u_{j,k} \ : \ j\in\Bbb J , \ k\in\Bbb K_j\rbrace$.
\end{enumerate}
\end{theorem}
\begin{proof}
(I). By (\ref{s2}), $\Lambda$ is a g-fusion frame (resp. g-fusion Bessel sequence, Parseval g-fusion frame) for $H$ if and only if $\lbrace u_{j,k} \ : \ j\in\Bbb J , \ k\in\Bbb K_j\rbrace$ is a frame (resp. Bessel sequence, Parseval  frame) for $H$. 

Assume that $\Lambda$ is a gf-Riesz basis for $H$ and $g_j\in H_j$. Thus, 
$$g_j=\sum_{k\in\Bbb K_j}c_{j,k}e_{j,k},$$
where $c_{j,k}\in\ell(\Bbb K_j)$. Note that, by (\ref{s2}), 
$$\lbrace f \ : \ \Lambda_j\pi_{W_j}f=0 , \ \ j\in\Bbb J\rbrace=\lbrace f \ : \ \langle f, u_{j,k}\rangle=0 , \ \ j\in\Bbb J , \ k\in\Bbb K_j\rangle\rbrace$$
and
\begin{equation*}
A\sum_{j\in\Bbb I}\Vert g_j\Vert^2\leq\big\Vert\sum_{j\in\Bbb I}v_j\pi_{W_j}\Lambda^*_j g_j \big\Vert^2\leq B\sum_{j\in\Bbb I}\Vert g_j\Vert^2
\end{equation*}
is equivalent to
\begin{equation*}
A\sum_{j\in\Bbb I}\sum_{k\in\Bbb K_j}\vert c_{j,k}\vert^2\leq\big\Vert\sum_{j\in\Bbb I}\sum_{k\in\Bbb K_j}c_{j,k}u_{j,k} \big\Vert^2\leq B\sum_{j\in\Bbb I}\sum_{k\in\Bbb K_j}\vert c_{j,k}\vert^2
\end{equation*}
for any finite $\Bbb I\subseteq\Bbb J$.
Thus, $\Lambda$ is a gf-Riesz basis if and only if $\lbrace u_{j,k} \ : \ j\in\Bbb J , \ k\in\Bbb K_j\rbrace$ is a Riesz basis.

Now, let $(W_j, \Lambda_j)$ be a gf-orthonormal basis for $H$ with respect to $\{v_j\}_{j\in\Bbb J}$. We get for any $j_1, j_2\in\Bbb J$, $k_1\in\Bbb K_{j_1}$ and $k_2\in\Bbb K_{j_2}$
\begin{align*}
\langle u_{j_1, k_1}, u_{j_2, k_2}\rangle&=\langle v_{j_1}\pi_{W_{j_1}}\Lambda^*_{j_1}e_{j_{1},k_{1}}, v_{j_2}\pi_{W_{j_2}}\Lambda^*_{j_2}e_{j_{2},k_{2}}\rangle\\
&=\delta_{j_1, k_1}\delta_{j_2, k_2}.
\end{align*}
So, $\lbrace u_{j,k} \ : \ j\in\Bbb J , \ k\in\Bbb K_j\rbrace$ is an orthonormal sequence. Moreovere
$$\Vert f\Vert^2=\sum_{j\in\Bbb J}v_j^2\Vert\Lambda_j\pi_{W_j}f\Vert^2=\sum_{j\in\Bbb J}\sum_{k\in\Bbb K_j}\vert\langle f, u_{j,k}\rangle\vert^2$$
for any $f\in H$. Hence $\lbrace u_{j,k} \ : \ j\in\Bbb J , \ k\in\Bbb K_j\rbrace$ is an orthonormal basis. For the opposite implication, we need only to prove that (\ref{o1}) holds. We have by (\ref{s3}) for each $j_1\neq j_2\in\Bbb J$, $g_{j_1}\in H_{j_1}$ and $g_{j_2}\in H_{j_2}$,
\begin{small}
\begin{eqnarray*}
\langle v_{j_1}\pi_{W_{j_1}}\Lambda^*_{j_1}g_{j_1},  v_{j_2}\pi_{W_{j_2}}\Lambda^*_{j_2}g_{j_2}\rangle=\big\langle\sum_{k_1\in\Bbb K_{j_1}}\langle g_{j_1}, e_{j_1,k_1}\rangle u_{j_1,k_1}, \sum_{k_2\in\Bbb K_{j_2}}\langle g_{j_2}, e_{j_2,k_2}\rangle u_{j_2,k_2}\big\rangle=0
\end{eqnarray*}
\end{small}
and for all $g_1, g_2\in H_j$
\begin{small}
\begin{eqnarray*}
\langle v_{j}\pi_{W_{j}}\Lambda^*_{j}g_{1},  v_{j}\pi_{W_{j}}\Lambda^*_{j}g_{2}\rangle=\big\langle\sum_{k_1\in\Bbb K_{j}}\langle g_{1}, e_{j,k_1}\rangle u_{j,k_1}, \sum_{k_2\in\Bbb K_{j}}\langle g_{2}, e_{j,k_2}\rangle u_{j,k_2}\big\rangle=\langle g_1, g_2\rangle.
\end{eqnarray*}
\end{small}
(II). By (\ref{s2}) and (\ref{s3}) we have for any $f\in H$
\begin{align*}
\sum_{j\in\Bbb J}v_j^2\pi_{W_j}\Lambda^*_j\Lambda_j\pi_{W_j}f&=\sum_{j\in\Bbb J}\sum_{k\in\Bbb K_j}\langle v_j\Lambda_j\pi_{W_j}f, e_{j,k}\rangle u_{j,k}\\
&=\sum_{j\in\Bbb J}\sum_{k\in\Bbb K_j}\Big\langle \sum_{k'\in\Bbb K_j}\langle f,u_{j,k'}e_{j,k'}\rangle, e_{j,k}\Big\rangle u_{j,k}\\
&=\sum_{j\in\Bbb J}\sum_{k\in\Bbb K_j}\langle f,u_{j,k}\rangle u_{j,k}.
\end{align*}
\end{proof}

\section{Perturbation of g-Fusion Frames}
In this section, we present some perturbation of g-fusion frames and review some results about them. First, we need the following which is proved in \cite{caz}.
\begin{lemma}\label{G1}
Let U be a Linear operator on a Banach space X and assume that there exist  $\lambda_{1},\lambda_{2}\in[0,1) $ such that 
 \begin{align*}
  & \Vert x-Ux\Vert \leq\lambda_{1}\Vert x\Vert+\lambda_{2}\Vert Ux\Vert
\end{align*}
for all $ x\in X$. Then $U$ is bounded and invertible. Moreovere,
 \begin{align*}
 \frac{1-\lambda_{1}}{1+\lambda_{2}}\Vert x\Vert\leq\Vert Ux\Vert\leq\frac{1+\lambda_{1}}{1-\lambda_{2}}\Vert x\Vert
 \end{align*}
 and
 \begin{align*}
 \frac{1-\lambda_{2}}{1+\lambda_{1}}\Vert x\Vert\leq\Vert U^{-1}x\Vert\leq\frac{1+\lambda_{2}}{1-\lambda_{1}}\Vert x\Vert
 \end{align*}
 for all $x\in X$.
\end{lemma}
\begin{theorem}\label{p1}
Let  $\Lambda:=(W_{j},\Lambda_{j},v_{j})$ be a g-fusion frame  for $H$ with bounds $A,B$ and $\lbrace\Theta_{j}\in\mathcal{B}(H, H_j)\rbrace_{j\in\Bbb J}$ be a sequence of operators  such that for any finite subset $\Bbb I\subseteq \Bbb J$ and for each $ f\in H$
\begin{small}
\begin{align*}
\Vert \sum_{j\in\Bbb I}v_{j}^2\big (\pi_{W_{j}}\Lambda^*_ j\Lambda_ j\pi_{W_{j}}f-&\pi_{W_{j}}\Theta^*_{j}\Theta_{j}\pi_{W_{j}}f\big)\Vert \leq
\lambda \Vert \sum_{j\in\Bbb I}v_{j}^2\pi_{W_{j}}\Lambda^*_j\Lambda_j\pi_{W_{j}} f\Vert\\
& +\mu \Vert \sum_{j\in\Bbb I}v_{j}^2\pi_{W_{j}}\Theta^*_j\Theta_j\pi_{W_{j}} f\Vert
 +\gamma (\sum_{j\in\Bbb I}v_{j}^2\Vert\Lambda_j\pi_{W_{j}} f\Vert^{2})^{\frac{1}{2}},
\end{align*} 
\end{small}
where $ 0\leq \max\lbrace \lambda+\frac{\gamma}{\sqrt{A}}, \mu\rbrace<1$, then $\Theta:=(W_j, \Theta_j, v_j)$ is a g-fusion frame for $H$ with bounds 
$$ A \frac{1-\big(\lambda+\frac{\gamma}{\sqrt{A}}\big)}{1+\mu} \ \  and \ \ B\frac{1+\lambda+\frac{\gamma}{\sqrt{B}}}{1-\mu}$$
\end{theorem}
\begin{proof}
Assume that  $\Bbb I\subseteq \Bbb J$ is a finite subset and  $f\in H$. We have 
\begin{small}
\begin{align*}
\Vert \sum_{j\in\Bbb I}v_{j}^2\pi_{W_{j}}\Theta^*_ j\Theta_ j\pi_{W_{j}}f\Vert &\leq \Vert \sum_{j\in\Bbb I}v_{j}^2\big (\pi_{W_{j}}\Lambda^*_ j\Lambda_ j\pi_{W_{j}}f-\pi_{W_{j}}\Theta^*_{j}\Theta_{j}\pi_{W_{j}}f\big)\Vert\\
&\ \ \ \ \ \ \ \ \ \ \ \ \ \ \ \ \ \ \  +\Vert \sum_{j\in\Bbb I}v_{j}^2 \pi_{W_{j}}\Lambda^*_ j\Lambda_ j\pi_{W_{j}}f\Vert\\
&\leq (1+\lambda)\Vert \sum_{j\in\Bbb I}v_{j}^2\pi_{W_{j}}\Lambda^*_ j\Lambda_ j\pi_{W_{j}}f\Vert +\mu\Vert \sum_{j\in\Bbb I }v_{j}^2\pi_{W_{j}}\Theta^*_ j\Theta_ j\pi_{W_{j}}f\Vert\\
& \ \ \ \ \ \ \ \ \ \ \ \ \ \ \ \ \ \ +\gamma(\sum_{j\in\Bbb I}v_{j}^2\Vert\Lambda_j\pi_{W_{j}} f\Vert^{2})^{\frac{1}{2}}.  
\end{align*}
\end{small}
Then
\begin{small}
\begin{align*}
 \Vert \sum_{j\in\Bbb I}v_{j}^2\pi_{W_{j}}\Theta^*_ j\Theta_ j\pi_{W_{j}}f\Vert&\leq \frac{1+\lambda}{1-\mu}\Vert \sum_{j\in\Bbb I}v_{j}^2\pi_{W_{j}}\Lambda^*_ j\Lambda_ j\pi_{W_{j}}f\Vert +\frac{\gamma}{1-\mu} ( \sum_{j\in\Bbb I}v_{j}^2\Vert\Lambda_j\pi_{w_{j}}f\Vert^{2})^\frac{1}{2}.
\end{align*}
\end{small}
Let $ S_{\Lambda}$ be the g-fusion frame operator of $ \Lambda$, then
\begin{align*}
\Vert \sum_{j\in\Bbb I}v_{j}^2\pi_{W_{j}}\Lambda^*_ j\Lambda_ j\pi_{W_{j}}f\Vert \leq\Vert S_{\Lambda}f\Vert
\leq B\Vert f \Vert.
\end{align*}
Therefore, for all $ f\in H $
\begin{align*}
 \Vert \sum_{j\in\Bbb I}v_{j}^2\pi_{W_{j}}\Theta^*_ j\Theta_ j\pi_{W_{j}}f\Vert\leq(\frac{1+\lambda}{1-\mu}\sqrt{B}+\frac{\gamma}{1-\mu})(\sum_{j\in\Bbb I}v_{j}^2\Vert\Lambda_ j\pi_{W_{j}}f\Vert^{2})^\frac{1}{2}<\infty.
\end{align*}
So, 
$\sum_{j\in\Bbb J} v_{j}^2\pi_{W_{j}}\Theta^*_ j\Theta_ j\pi_{W_{j}}f$
is unconditionally convergent.
Let 
\begin{align*}
S_\Theta&:H\longrightarrow H \\ 
S_\Theta(f)&=\sum_{j\in\Bbb J} v_{j}^2\pi_{W_{j}}\Theta^*_ j\Theta_ j\pi_{W_{j}}f.    
\end{align*}
 $ S_\Theta $ is well defined and bounded operator with
\begin{align*}
 \Vert S_\Theta\Vert\leq\frac{1+\lambda}{1-\mu} B +\frac {\gamma \sqrt {B}}{1-\mu}
\end{align*}
 and for each $f\in H$, we have
\begin{align*}
\sum_{j\in\Bbb J} v_{j}^2\Vert\Theta_ j\pi_{W_{j}}f\Vert^{2}=\langle S_{\Theta}f, f\rangle \leq\Vert S_{\Theta}\Vert\Vert f\Vert^{2}.
\end{align*}
It follows that $ \Theta:=(W_{j},\Theta_{j}, v_{j})$ is a g-fusion Bessel sequence for $H$. Thus, we obtain by the hypothesis
\begin{align*}
\Vert S_{\Lambda}f-S_{\Theta}f\Vert\leq\lambda\Vert S_{\Lambda}f\Vert + \mu\Vert S_{\Theta}f\Vert +\gamma (\sum_{j\in\Bbb I}v_{j}^2\Vert\Lambda_{j}\pi_{w_{j}}f\Vert)^2)^\frac{1}{2}.
\end{align*}
Therefore, by (\ref{sf1}) 
\begin{align*}
\Vert f-S_{\Theta}S_{\Lambda}^{-1}f\Vert&\leq\lambda\Vert f\Vert+\mu\Vert S_{\Theta}S_{\Lambda}^{-1}f\Vert+\gamma\big(\sum_{j\in\Bbb J}v_J^2\Vert\Lambda_j\pi_{W_j}S_{\Lambda}^{-1}f\Vert^2\big)^{\frac{1}{2}}\\
&\leq\Big(\lambda+\frac{\gamma}{\sqrt{A}}\Big)\Vert f\Vert+\mu\Vert S_{\Theta}S_{\Lambda}^{-1}f\Vert.
\end{align*}
Since $0\leq \max\lbrace \lambda+\frac{\gamma}{\sqrt{A}}, \mu\rbrace<1$, then by Lemma \ref{G1}, $S_{\Theta}S_{\Lambda}^{-1}$ and consequently $S_{\Theta}$ is invertible and we get
$$\Vert S_{\Theta}^{-1}\Vert\leq\Vert S_{\Lambda}^{-1}\Vert \Vert S_{\Lambda}S_{\Theta}^{-1}\Vert\leq\frac{1+\mu}{A \big(1-(\lambda+\frac{\gamma}{\sqrt{A}})\big)}.$$
So, the proof is completed.
\end{proof}
\begin{corollary}
The optimal lower and upper bounds of $\Theta$ which defined in Theorem \ref{p1} are $\Vert S_{\Theta}^{-1}\Vert^{-1}$ and $\Vert S_{\Theta}\Vert$, respectively.
\end{corollary}
\begin{corollary}\label{p2}
Let  $\Lambda$ be a g-fusion frame  for $H$ with bounds $A,B$ and $\lbrace\Theta_{j}\in\mathcal{B}(H, H_j)\rbrace_{j\in\Bbb J}$ be a sequence of operators. If there exists a constant $0<R<A$ such that 
\begin{align*}
 \sum_{j\in\Bbb J}v_{j}^2\Vert\pi_{W_{j}}\Lambda^*_ j\Lambda_ j\pi_{W_{j}}f-\pi_{W_{j}}\Theta^*_{j}\Theta_{j}\pi_{W_{j}}f\Vert \leq R\Vert f\Vert
\end{align*} 
for all $f\in H$, then $\Theta:=(W_j, \Theta_j, v_j)$ is a g-fusion frame for $H$ with bounds
$$A-R \ \ \ \ and \ \ \ \ \min\{B+R\sqrt{\frac{B}{A}}, R+\sqrt{B}\}.$$
\end{corollary}
\begin{proof}
It is easy to check that $\sum_{j\in\Bbb J}v_j^2\pi_{W_{j}}\Theta^*_{j}\Theta_{j}\pi_{W_{j}}f$ is converges for any $f\in H$. Thus, we obtain for each $f\in H$
\begin{align*}
\sum_{j\in\Bbb J}v_j^2\Vert \pi_{W_j}\Lambda^*_j\Lambda_j\pi_{W_j}f-\pi_{W_j}\Theta^*_j\Theta_j\pi_{W_j}f\Vert&\leq R\Vert f\Vert\\
&\leq\frac{R}{\sqrt{A}}\big(\sum_{j\in\Bbb J}v_j^2\Vert\Lambda_j\pi_{W_j}f\Vert^2\big)^{\frac{1}{2}}
\end{align*}
and also
\begin{align*}
\sum_{j\in\Bbb J}v_j^2\Vert \pi_{W_j}\Theta^*_j\Theta_j\pi_{W_j}f\Vert&\leq R\Vert f\Vert+\sum_{j\in\Bbb J}v_j^2\Vert\pi_{W_j}\Lambda^*_j\Lambda_j\pi_{W_j}f\Vert\\
&\leq(R+\sqrt{B})\Vert f\Vert.
\end{align*}
By using Theorem \ref{p1} with  $\lambda=\mu=0$ and $\gamma=\frac{R}{\sqrt{A}}$, the proof is completed.
\end{proof}
The following is another version of perturbation of g-fusion frames.
\begin{theorem}
Let $\Lambda$ be a g-fusion frame for $H$ with bounds $A,B$ and $\lbrace\Theta_{j}\in\mathcal{B}(H, H_j)\rbrace_{j\in\Bbb J}$ be a sequence of operators  such that for any finite subset $\Bbb I\subseteq \Bbb J$ and for each $ \{f_j\}_{j\in\Bbb J}\in\mathscr{H}_2$,
\begin{align*}
\Vert \sum_{j\in\Bbb I}v_{j}\big (\pi_{W_{j}}\Lambda^*_ jf_j-\pi_{W_{j}}\Theta^*_{j}f_j\big)\Vert &\leq
\lambda \Vert \sum_{j\in\Bbb I}v_{j}(\pi_{W_{j}}\Lambda^*_j f_j)\Vert+\\
& \ \ \ \ \ +\mu \Vert \sum_{j\in\Bbb I}v_{j}(\pi_{W_{j}}\Theta^*_j f_j)\Vert
 +\gamma (\sum_{j\in\Bbb I}\Vert f_j\Vert^{2})^{\frac{1}{2}},
\end{align*} 
where $ 0\leq \max\lbrace \lambda+\frac{\gamma}{\sqrt{A}}, \mu\rbrace<1$, then $\Theta:=(W_j, \Theta_j, v_j)$ is a g-fusion frame for $H$ with bounds 
$$ A \Big(\frac{1-\big(\lambda+\frac{\gamma}{\sqrt{A}}\big)^2}{1+\mu}\Big) \ \  and \ \ B\Big(\frac{1+\lambda+\frac{\gamma}{\sqrt{B}}}{1-\mu}\Big)^2.$$
\end{theorem} 
\begin{proof}
Let $ \{f_j\}_{j\in\Bbb J}\in\mathscr{H}_2$, then
\begin{small}
\begin{align*}
\Vert\sum_{j\in\Bbb I}v_j\pi_{W_{j}}\Theta^*_{j}f_j\Vert&\leq\Vert \sum_{j\in\Bbb I}v_{j}\big (\pi_{W_{j}}\Lambda^*_ jf_j-\pi_{W_{j}}\Theta^*_{j}f_j\big)\Vert
+\Vert\sum_{j\in\Bbb I}v_j\pi_{W_{j}}\Lambda^*_ jf_j\Vert\\
&\leq(1+\lambda)\Vert \sum_{j\in\Bbb I}v_{j}\pi_{W_{j}}\Lambda^*_j f_j\Vert+\\
& \ \ \ \ \ \ \ \ \ \ \ \ \ \ \ +\mu \Vert \sum_{j\in\Bbb I}v_{j}\pi_{W_{j}}\Theta^*_j f_j\Vert
 +\gamma\big(\sum_{j\in\Bbb I}\Vert f_j\Vert^{2}\big)^{\frac{1}{2}}.
\end{align*}
\end{small}
Hence, 
\begin{align*}
\Vert\sum_{j\in\Bbb I}v_j\pi_{W_{j}}\Theta^*_{j}f_j\Vert&\leq\frac{1+\lambda}{1-\mu}\Vert \sum_{j\in\Bbb I}v_{j}\pi_{W_{j}}\Lambda^*_j f_j\Vert+\frac{\gamma}{1-\mu}\big(\sum_{j\in\Bbb I}\Vert f_j\Vert^{2}\big)^{\frac{1}{2}}\\
&\leq\Big(\frac{1+\lambda}{1-\mu}\sqrt{B}+\frac{\gamma}{1-\mu}\Big)\Big(\sum_{j\in\Bbb I}\Vert f_j\Vert^{2}\Big)^{\frac{1}{2}}.
\end{align*}
Let
\begin{align*}
T_{\Theta}&:H\longrightarrow \mathscr{H}_2\\
T_{\Theta} \{f_j\}_{j\in\Bbb J}&=\sum_{j\in\Bbb I}v_j\pi_{W_{j}}\Theta^*_{j}f_j.
\end{align*}
Therefore, $T_{\Theta}$ is a well-defined and bounded. Then, by Theorem \ref{2.2}, $\Theta$ is a g-fusion Bessel sequence. Suppose that $G:=T_{\Theta}T^*_{\Lambda}S^{-1}_{\Lambda}$. Then, we get by the hypotesis and (\ref{s1}) for $f_j:=\Lambda_j\pi_{W_j}S^{-1}_{\Lambda}f$,
\begin{align*}
\Vert f-Gf\Vert&\leq\lambda\Vert f\Vert+\mu\Vert Gf\Vert+\gamma\big(\sum_{j\in\Bbb J}v_j^2\Vert\Lambda_j\pi_{W_j}S^{-1}_{\Lambda}f\Vert^2\big)^\frac{1}{2}\\
&\leq\big(\lambda+\frac{\gamma}{\sqrt{A}}\big)\Vert f\Vert+\mu\Vert Gf\Vert.
\end{align*}
Since $ 0\leq \max\lbrace \lambda+\frac{\gamma}{\sqrt{A}}, \mu\rbrace<1$, by Lemma \ref{G1}, $G$ and consequently $T_{\Theta}T^*_{\Lambda}$ is invertible and
$$\Vert G^{-1}\Vert\leq\frac{1+\mu}{1-(\lambda+\frac{\gamma}{\sqrt{A}})}.$$
Now, let $f\in H$ and we have
\begin{align*}
\Vert f\Vert^4&=\vert\langle GG^{-1}f, f\rangle\vert^2\\
&=\big\vert\big\langle\sum_{j\in\Bbb J}v_j^2\pi_{W_j}\Theta^*_j\Lambda_j\pi_{W_j}S^{-1}_{\Lambda}(G^{-1}f), f\big\rangle\big\vert^2\\
&\leq\sum_{j\in\Bbb J}\Vert\Lambda_j\pi_{W_j}S^{-1}_{\Lambda}(G^{-1}f)\Vert^2.\sum_{j\in\Bbb J}v_j^2\Vert \Theta_j\pi_{W_j}f\Vert^2\\
&\leq A^{-1}\Vert G^{-1}f\Vert^2\sum_{j\in\Bbb J}v_j^2\Vert \Theta_j\pi_{W_j}f\Vert^2
\end{align*}
and this copmletes the proof.
\end{proof}
\begin{theorem}
Let  $\Lambda$ be a g-fusion frame  for $H$ with bounds $A,B$ and $\lbrace\Theta_{j}\in\mathcal{B}(H, H_j)\rbrace_{j\in\Bbb J}$ be a sequence of operators. If there exists a constant $0<R<A$ such that 
\begin{align*}
 \sum_{j\in\Bbb J}v_{j}^2\Vert\Lambda_ j\pi_{W_{j}}f-\Theta_{j}\pi_{W_{j}}f\big)\Vert^2 \leq R\Vert f\Vert^2
\end{align*} 
for all $f\in H$, then $\Theta:=(W_j, \Theta_j, v_j)$ is a g-fusion frame for $H$ with bounds
$$(\sqrt{A}-\sqrt{R})^2 \ \ \ \ and \ \ \ \ (\sqrt{R}+\sqrt{B})^2.$$
\end{theorem}
\begin{proof}
Let $f\in H$. We can write
\begin{align*}
\Vert\{v_j\Theta_j\pi_{W_j}f\}_{j\in\Bbb J}\Vert_2&\leq\Vert\{v_j\Theta_j\pi_{W_j}f-v_j\Lambda_j\pi_{W_j}f\}_{j\in\Bbb J}\Vert_2+\Vert\{v_j\Lambda_j\pi_{W_j}f\}_{j\in\Bbb J}\Vert_2\\
&\leq(\sqrt{R}+\sqrt{B})^2\Vert f\Vert^2
\end{align*}
and also
\begin{align*}
\Vert\{v_j\Theta_j\pi_{W_j}f\}_{j\in\Bbb J}\Vert_2&\geq\Vert\{v_j\Lambda_j\pi_{W_j}f\}_{j\in\Bbb J}\Vert_2-\Vert\{v_j\Theta_j\pi_{W_j}f-v_j\Lambda_j\pi_{W_j}f\}_{j\in\Bbb J}\Vert_2\\
&\geq(\sqrt{A}-\sqrt{R})^2\Vert f\Vert^2.
\end{align*}
Thus, these complete the proof.
\end{proof}

\section{Conclusions}
In this paper, we could transfer some common properties of g-frames to g-fusion frames. First, we reviewed gf-Riesz and orthonormal bases and showed that the weights have  basic roles in their definitions. Afterwards,  characterizations of these frames were presented in Section 4 by Sun's method by using the operator (\ref{s1}). Finally, in Section 5, we introduced the perturbation of g-fusion frames and reviewed some useful statements.

\end{document}